\documentclass[12pt, notitlepage]{amsart}

\usepackage[linktocpage=true,colorlinks,citecolor=blue,linkcolor=blue,urlcolor=blue,pagebackref=true]{hyperref}
\usepackage[abbrev]{amsrefs}
\usepackage{amssymb}
\usepackage{amsmath}
\usepackage{amsthm}
\usepackage{amsfonts}
\usepackage{comment}
\usepackage[autostyle]{csquotes}
\usepackage[colorinlistoftodos]{todonotes}

\numberwithin{equation}{section}
\def\Q{\mathbb Q}
\def\Z{\mathbb Z}

\def\P{\mathbb P}
\def\F{\mathbb F}

\def\E{\mathbb E}

\def\ee{\varepsilon}
\def\wt{\widetilde}
\newcommand{\wh}{\widehat}

\def\gcd{\operatorname{gcd}}

\newcommand{\cO}{\mathcal{O}}
\newcommand{\cP}{\mathcal{P}}

\newtheorem{theorem}{Theorem}[section]
\newtheorem{lemma}[theorem]{Lemma}
\newtheorem{proposition}[theorem]{Proposition}

\newtheorem{corollary}[theorem]{Corollary}

\newtheorem{problem}[theorem]{Problem}

\theoremstyle{remark}

\theoremstyle{definition}

\theoremstyle{remark}

\usepackage{xcolor}

\numberwithin{equation}{section}

\renewcommand{\a}{\alpha}
\renewcommand{\b}{\beta}

\newcommand{\e}{\varepsilon}

\begin{document}
	\title[Random polynomials with multiplicative coefficients]{A random polynomial with multiplicative coefficients is almost surely irreducible}
   \author{P\'eter P. Varj\'u}
\address{Centre for Mathematical Sciences, Wilberforce Road, Cambridge CB3 0WA, UK}
\email{pv270@dpmms.cam.ac.uk} 
	\author{Max Wenqiang Xu}
\thanks{MWX was supported by a Simons Junior Fellowship from Simons Foundation.}
\address{Yau Mathematical Sciences Center, Tsinghua University, Beijing, China}
\email{maxxu1729@gmail.com}

	\begin{abstract}
		Assume that the Riemann hypothesis holds for Dedekind zeta functions.
Under this assumption,
we prove that a degree $d$ polynomial with random multiplicative $\pm1$ coefficients is
irreducible in $\Z[x]$ with probability $1-O(d^{-1/2+\e})$.
	\end{abstract}
	\maketitle

 \section{Introduction}

The question of how likely it is that a random polynomial in $\Z[x]$
is irreducible has a long history.
The first studied model was where the degree of the polynomials is a fixed
number and the coefficients are sampled independently and uniformly from
growing intervals.
The is less relevant to our paper, and we only refer to the recent
breakthrough \cite{Bha} and its references.

Another setting that has gained momentum more recently is where the
coefficients are sampled independently from a fixed law where the constant coefficient is nonzero and the degree
of the polynomials is growing.
A sequence of papers \cite{BSK}, \cite{BSKK}, \cite{Baz} established that such random polynomials are irreducible
with probability tending to $1$ if the common law of the coefficients
is uniform enough modulo $4$ primes, in particular when the coefficients are
uniformly distributed on $35$ consecutive elements.
See also \cite{BSHKP} for results about $\pm1$ coefficients and special degrees.
Using a different method, and assuming the Riemann hypothesis for Dedekind zeta
functions, \cite{BV19} proved that the probability that a random polynomial is
irreducible tends to $1$ requiring only the necessary condition that the constant
coefficient is not $0$.
This conditionally solved a conjecture of Odlyzko and Poonen \cite{OP93} and the method also yields better estimates for the probability that the
random polynomial is reducible.

In another direction, \cite{Ebe} and \cite{FJSS} proved irreducibility of the characteristic polynomial of random matrices with high probability.

In this paper, we consider other models where the coefficients of the random polynomial are not independent.
We define a sequence $X_n$ of $\pm1$ valued random variables for $n\in\Z_{\ge1}$
as follows. We let $X_1=1$ with probability $1$.
For primes $p$, we let $X_p$ be independent uniform random variables taking
$\pm1$ values. We consider two models. One is that for
 $n\in\Z_{\ge 2}$ we let $X_n=X_{p_1}\cdots X_{p_k}$, where
$n=p_1\cdots p_k$ is the prime factorization of $n$.
The other model is that $X_n$ is supported only on square-free integers $n$
defined in the same way,
and if $n$ is not square-free, we set $X_n=0$. 
For $d\in\Z_{\ge 0}$, and $\{X_n\}_{1\le n \le d}$ being either one of the above two models,
we define a random polynomial with multiplicative coefficients as
\[
P_d(X)=X_1x+X_2 x^{2}+\ldots+ X_{d} x^{d}.
\]

The main result of the paper is the following.

\begin{theorem} \label{thm: random}
	Suppose that the Riemann hypothesis holds for the Dedekind zeta functions
	of all number fields.
	Then for every $\e>0$, there is a constant $C=C(\e)$ such that 
	\[\P[P_d(x)/x~\text{is irreducible over $\Z$}  ] \ge 1- C d^{-1/2+\e}. \]
\end{theorem}
We remark that in the model where $X_n$ is supported on square-free $n$ only, our bound is close to be sharp, as it is proved in \cite{AACKX25} that $x=1$ is a root with probability at least $Cd^{-1/2-\ee}$.

Polynomials with multiplicative coefficients are of great interest in number theory.
The study of their values on the unit circle has a vast literature. See \cite{BNR} and \cite{Har} for recent work in the setting of polynomials with
random multiplicative coefficients.

The question of irreducibility was recently studied in the setting of
Fekete polynomials in \cite{MNT23} and \cite{MNT24}.
For a prime $p$, the Fekete polynomial $F_p$ is defined
as
\[
F_p(x)=\sum_{a=1}^{p-1}\Big(\frac{a}{p}\Big) x^a,
\]
where $\big(\frac{a}{p}\big)$ denotes the Legendre symbol.
The authors of \cite{MNT23} have made the conjecture that $F_p(x)$
is a product of linear factors corresponding to possible roots at $-1,0,1$
and an irreducible polynomial for all $p$.

Motivated by this, we pose the following problem.
\begin{problem}\label{prb:Fekete}
Let $f:\Z_{>0}\to \Z_{>0}$ be a function such that for all $d$, there is at least
one prime with $d<p\le f(d)$.
For $d\in\Z_{>0}$, let $p$ be a random prime in $(d,f(d)]$ sampled uniformly
and let
\begin{equation}\label{eq:Fekete-def}
F_{d,f}(x)= \sum_{a=1}^{d} \Big(\frac{a}{p}\Big) x^a
\end{equation}
be a random polynomial.
What is the asymptotic behaviour of the probability that $F_{d,f}$ is irreducible
after removing possible linear factors?
\end{problem}

The conjecture in \cite{MNT23} predicts that the probability in question is
$1$ when $d+1$ is a prime and $f(d)=d+1$.
If we allow $f(d)>2^{2\pi(d)}d^4$, then it is an immediate consequence of our
main result and known results about the distribution of the  Legendre symbol
that under the Riemann hypothesis for Dedekind zeta functions,
$F_{d,f}$ is irreducible with high probability.
Therefore, Problem \ref{prb:Fekete} interpolates between the results
of this paper and what is conjectured for Fekete polynomials.
The slower the function $f$ grows, the randomness in the model decreases, and the problem
likely becomes harder.
We believe the problem is a good testing ground for new ideas, and hope it will stimulate
further progress.

\begin{corollary}\label{cr:pd}
	Let $F_{d, f}(x)$ be defined as in \eqref{eq:Fekete-def}
	and suppose $f(d)>2^{2\pi(d)}d^4$.
	Suppose that the Riemann hypothesis holds for Dedekind zeta functions
	of all number fields.
	Then for all $\e>0$, there is $C=C(\e)$ such that
	\[
	\P[{F_{d,f}(x)/x}~\text{is irreducible over $\Z$}]
	\ge 1 - C d^{-1/2+\e}.
	\]
\end{corollary}

It would be interesting to see what the behaviour is when the range from which
the prime is sampled is shorter.

\subsection{Outline of the proof}

The proof of Theorem \ref{thm: random} follows the strategy of \cite{BV19}, which
we briefly recall.
Fix a polynomial $P$, and choose a random prime $q$ with a suitably chosen
probability distribution.
It is a consequence of the prime ideal theorem that if $P$ is irreducible, then
it has on average $1$ root in $\F_q$.
For different irreducible polynomials these roots rarely coincide, so we can
deduce that
\begin{align}\label{eq:PIT}
\{\text{number of distinct irreducible} &\text{ factors of $P$}\}\nonumber\\
&\approx \E_q[\text{number of roots of $P$ in $\F_q$}],
\end{align}
where $P$ is a fixed polynomial and the averaging is over a random prime $q$.
This is where the Riemann hypothesis is used to control the error term.
In particular, it will be important for us that the approximation is valid even if we average
over a relatively small range of primes.

If we take a random polynomial $P$, and show that it has on average
$1$ root in $\F_q$, now $P$ and $q$ are both random, then it follows
that $P$ is a power of a single irreducible polynomial with high probability.
To show this, we fix a prime $q$ and a residue $a\in\F_q$, and show that the
value $P(a)$ is equidistributed in $\F_q$ for our random $P$.
In particular, $P(a)=0\in\F_q$ will occur with probability approximately $1/q$.
Summing this up for $a$ and averaging over $q$ will give the required result. 

In the setting of \cite{BV19}, the equidistribution of $P(a)$ in $\F_q$
is related to a Markov chain introduced by Chung, Diaconis and Graham \cite{CDG87}.
Due to the dependence of the coefficients, the equidistribution problem
cannot be described by a Markov chain in our setup.

Proving equidistribution is the main new contribution of our paper. We do this by conditioning on the values of the coefficients $X_p$ for primes $p<d/2$, and use that the coefficients corresponding to the
remaining primes are independent from each other and from the coefficients
whose values are influenced by the smaller primes.
We build on \cite{Kon} (also used in \cite{BV19})
to prove equidistribution of
\[
\sum_{p\ge d/2\text{ prime}} X_p a^{p-1}.
\]
The key difference is that we are summing over the primes, and the argument
in \cite{BV19} requires an arithmetic progression.
We bypass this issue by finding many disjoint short arithmetic progressions
in the primes using known results from \cite{GT}, \cite{TZ}, \cite{Sha}
and then apply a version of the argument in \cite{BV19} for these.
This idea has been inspired by \cite{Bou}.

This allows us to prove equidistribution of $P(a)$ for most values of $a$.
For some of the remaining values, equidistribution may fail.
It certainly does for $a=-1,0,1$ and possibly also for some other low degree
roots of unity.
For the exceptional residues, we only prove an upper bound of the form
$\P[P(a)=0]\le C d^{-1/2+\e}$ using a classical Littlewood-Offord type bound.
This is where the error term in Theorem \ref{thm: random} comes from.
Here a more precise analysis may yield a better bound giving a more precise
estimate for the probability that $P$ is reducible.
It may be possible to adapt the arguments in \cite{BV19} to give stronger
estimates for most of the exceptional residues, and the error term
may potentially be dominated by the probability that $P$ is divisible by $x-1$
or $x+1$. 
We do not pursue this question.

After this, it remains to show that $P$ is not a proper power of an irreducible
polynomial with high probability.
We have not been able to adapt the corresponding argument in \cite{BV19}
to our setting.
Instead, we extend the equidistribution result to the pair
$(P(a),P'(a))$.
This allows us to show that $P$ has very few
double roots in $\F_q$ on average (much less than $1$),
and then the main argument can
be used to rule out repeated factors of $P$ in $\Z[x]$.

\subsection{Organization of the paper}

In Section \ref{sc:PIT}, we formulate a statement that makes \eqref{eq:PIT}
precise.
This is a straightforward adaptation of \cite{BV19}, but the result formulated
in \cite{BV19} cannot be applied as a black box.
Section \ref{sc:equidistribution} contains our equidistribution estimate for
random walks with steps $\pm a^i$ where $i$ runs through an index set
that contains sufficiently many disjoint arithmetic progressions.
We use these estimates to derive bounds for the expected number of roots
of $P$ in a finite field in Section \ref{sc:number-of-roots}.
We finish the paper by completing the proofs of the main results in Sections
\ref{sc:proof-thm} and \ref{sc:proof-cor}.

\subsection{Notation}
The letters $c$ and $C$ denote positive constants whose values may change
from one occurrence to the other.

\subsection{Acknowledgement}
We would like to thank Jacob Fox, Julian Sahasrabudhe, Peter Sarnak and Joni Ter\"av\"ainen for helpful discussions on the subject of this paper. We also thank the referee for helpful comments. Part of the work was done during several visits of MWX at the University of Cambridge, and the warm hospitality is greatly appreciated.

\section{Expected number of roots of a polynomial in a random finite field}
\label{sc:PIT}

Given a number field $K$, we denote by $\cO_K$ its ring of integers and
write $\zeta_K$ for its Dedekind zeta function.
We write $A_K(n)$ for the number of prime ideals $\mathfrak{p}\subset\cO_K$
with norm $N_{K/\Q}(\mathfrak{p})=n$.

For a number $X>1$, we write
\[
h_X(u)=
\begin{cases}
2\exp(-X) & \text{if $u\in(X-\log 2,X]$,}\\
0 &\text{otherwise}.
\end{cases}
\]

Given a polynomial $P\in\Z[x]$ and a rational prime $q$,
we write $B_P(q)$ for the number of distinct roots of $P$
in $\F_q$.
We write $\wt P$ for the product of the irreducible factors
of $P$ in $\Z[x]$, and we write $\Delta_P$ for the discriminant of $P$.
Given an irreducible polynomial $Q\in\Z[x]$, we write $K_Q=\Q(\a)$,
where $\a$ is a root of $Q$.

The purpose of this section is to prove the following result.

\begin{proposition}\label{pr:number-distinct-factors}
Let $d, M\in\Z_{> 1}$.
Let $P\in\Z[x]$ be a polynomial of degree at most $d$
with coefficients of absolute value at most $M$.
Suppose that for every irreducible factor $Q$ of $P$,
RH holds for $\zeta_{K_Q}$.
Let $X\ge 1$.
Then
\begin{align*}
\sum_{\text{$q$ prime}} B_P(q)\log(q)h_X(\log q)
&=|\{\text{distinct irreducible factors of $P$}\}|\\
&\phantom{=}+ O(d^2 X^2\log (dM)\exp(-X/2)).
\end{align*}
The implied constant is absolute.
\end{proposition}

The proof of this follows \cite{BV19}*{Proposition 19}.
We begin by recalling a quantitative version of the prime ideal theorem under
the Riemann hypothesis.
This is standard, and this precise formulation can be found in
\cite{BV19}*{Proposition 9}.

\begin{proposition}\label{pr:prime-number-theorem}
Let $K$ be a number field with discriminant $\Delta$, and suppose
RH holds for $\zeta_K$.
Let $X>1$.
Then
\[
\sum_{\text{$q$ prime}} A_K(q)\log(q) h_X(\log q)
=1+O(X^2\max(\log|\Delta|,1)\exp(-X/2)),
\]
where the implied constant is absolute.
\end{proposition}

\begin{lemma}\label{lm:A-B}
Let $P\in\Z[x]$, and let $q$ be a rational prime with $q\nmid\Delta_{\wt P}$.
Then
\begin{equation}\label{eq:A-B}
B_P(q)=\sum_{\text{$Q|P$ irreducible}} A_{K_Q}(q).
\end{equation}
\end{lemma}

This lemma is standard.
It is closely related to the $m=1$ case of \cite{BV19}*{Proposition 16}.
The difference compared to that result is that there the roots of certain
exceptional polynomials are not counted in $B_P(p)$ and they are also not
counted on the right hand side of \eqref{eq:A-B}.
While this is not formally permitted in \cite{BV19}, the proof works verbatim
if we take the empty set for the exceptional polynomials.

\begin{proof}[Proof of Proposition \ref{pr:number-distinct-factors}]
We first estimate $\Delta_P$.
It is represented by a determinant of size $2d-1$ with entries bounded by $dM$
(divided by the leading coefficient).
Therefore,
\[
|\Delta_P|\le (2d-1)^{2d-1} (dM)^{2d-1}\le (2d)^{4d} M^{2d}.
\]
This is also an upper bound for $\Delta_{\wt P}$ and
the discriminants for all the number fields that
we may obtain by adjoining a root of $P$ to $\Q$.

If $q\nmid \Delta_{\wt P}$ is a prime, then
\[
B_P(q)=\sum_{\text{$Q|P$ irreducible}} A_{K_Q}(q)
\]
by Lemma \ref{lm:A-B}.
If $q|\Delta_{\wt P}$, then we just use the trivial bounds $0 \le B_P(q) \le d$, and $0\le \sum_{\text{$Q|P$ irreducible}} A_{K_Q}(q) \le d$ to deduce that
\begin{equation}\label{eqn: q}
\Big|B_P(q)-\sum_{\text{$Q|P$ irreducible}} A_{K_Q}(q)\Big|\le d, 
\end{equation}
and we have at most
\[
\frac{\log |\Delta_{\wt P}|}{X-\log 2}\le \frac{Cd\log (dM)}{X}
\]
prime factors $q$ of $\Delta_{\wt P}$ with $\log q \in (X-\log 2, X]$ for which we need to use
\eqref{eqn: q}.

Therefore,
\[
\sum_{\text{$q$ prime}}\Big|B_P(q)-\sum_{\text{$Q|P$ irreducible}} A_{K_Q}(q)\Big|
\log (q) h_X(\log q)
\le \frac{C d^2\log(dM)}{\exp(X)}
\]
for an absolute constant $C$.

Using Proposition \ref{pr:prime-number-theorem} to estimate the sums of $A_{K_Q}(q)$,
we get
\begin{align*}
\Big|\sum_{\text{$q$ is prime}} &B_P(q)\log(q)h_X(\log q)
-|\{\text{distinct irreducible factors of $P$}\}|\Big|\\
&\le Cd X^2\log((2d)^{4d}M^{2d})\exp(-X/2)+Cd^2\log(dM)\exp(-X)\\
&\le C d^2 X^2\log (dM)\exp(-X/2).
\end{align*}
\end{proof}

\section{Equidistribution estimate}
\label{sc:equidistribution}

We denote by $M(P)$ the Mahler measure of a polynomial $P\in\Z[x]$.
Let $l\in\Z_{\ge 1}$ and let $q$ be a prime.
We say that a polynomial $P$ is $(l,q)$-exceptional if $\deg(P)\le l$ and $M(P)\le q^{1/(l+1)^2}$.
If $q$ is a prime, then an element of $\F_q$ is $l$-exceptional if it
is the root of an $(l,q)$-exceptional polynomial.

The purpose of this section is to prove the following result and a weaker
estimate that is valid for all non-zero residues, which we formulate at the
end of the section.

\begin{proposition}\label{pr:unexceptional}
Let $l,d\in\Z_{\ge 3}$ and let $q$ be a prime that is suitably large in
terms of $l$.
Let $K\in\Z_{\ge1}$.
Let $a\in\F_q$ be an element such that $a^k$
is not $l$-exceptional for any $k=1,\ldots,K$.
Let $I\subset [0,\ldots,d]$ be a set that contains $q^{5/(l+1)}$ pairwise disjoint
arithmetic progressions of length $3l^3$ with common difference at most $K$.
Let $X_i$ be independent uniform $\pm1$ valued random variables, and
let
\[
Y=\sum_{i\in I}X_i(a^i,ia^{i-1})
\]
be a random element of $\F_q^2$.
Then
\[
\Big|\P[Y=x]-\frac{1}{q^2}\Big|< q^{-10}
\]
for all $x\in\F_q^2$.
\end{proposition}

The strategy of the proof is to estimate the Fourier coefficients
of $Y$.
For an element $x\in\F_q$, we write $|x|$ for the smallest absolute
value of an integer in the residue class of $x$.
We will show that if $a^k$ is not an exceptional residue and
$(\xi_1,\xi_2)\in\F_q^2\backslash(0,0)$, then
$|\xi_1 a^j+\xi_2 j a^{j-1}|$ cannot be small for all values of $j$
in a suitably long arithmetic progression of step size $k$.
This is the content of the next lemma.
Once we have this, we may use the assumption that $I$ contains
many disjoint arithmetic progressions to find many indices $j$ for which
$|\xi_1 a^j+\xi_2 j a^{j-1}|$ is not small.
This will allow us to estimate the Fourier transform using a product
formula that follows from the independence of $X_j$.

\begin{lemma}\label{lm:large-element}
Let $l\in\Z_{\ge 3}$, and let $r>l^2$ be a prime.
Let $q$ be a prime that is sufficiently large in terms of $l$, let $a\in\F_q$, and
let $(\xi_1,\xi_2)\in\F_q^2\backslash(0,0)$.
Let $k\in\Z_{> 0}$.
	Suppose that we have
	\[
	|\xi_1 a^{jk+j_0}+\xi_2(jk+j_0)a^{jk+j_0-1}|<\frac{q^{1-2/(l+1)}}{l+1}
	\]
	for all $j=0,\ldots, l(r+1)$.
	Then $a^k$ is an $l$-exceptional residue.
\end{lemma}

The proof of this lemma follows an argument of Konyagin \cite{Kon}.
Writing $b_j$ for the integer with the smallest absolute value
in the residue class of $\xi_1 a^{jk+j_0}+\xi_2(jk+j_0)a^{jk+j_0-1}$,
we will show that under the assumption of the lemma, $b_j$ satisfies
two linear recurrence relations.
We will use this to show that $b_j$ also satisfies the linear recurrence
corresponding to the greatest common divisor of the polynomials associated
to the original recurrences, and hence
$a$ is a root of this polynomial.
One of the polynomials will be used to control the degree, while the other will be
used to control the Mahler measure of the greatest common divisor.

The following simple lemma will be used to construct polynomials such that
$b_j$ satisfies the corresponding linear recurrences $\mod q$, and also at the
end of the proof to conclude that $a$ is a root of the greatest common divisor.

\begin{lemma}\label{lm:recurrence}
Let $a\in\F_q$
and let $\a_0,\ldots,\a_l\in\F_q$.
Consider the equations
\begin{equation}\label{eq:recurrence}
\sum_{j=0}^{l}\a_j(\xi_1 a^{j+j_0}+\xi_2(j+j_0)a^{j+j_0-1} )=0,
\end{equation}
where $\xi_1,\xi_2\in\F_q$
for $j_0\in\Z_{\ge 0}$ with the conventions $0\cdot 0^{-1}=0$ and $0^0=1$.

Then the following hold.
\begin{enumerate}
	\item If equation \eqref{eq:recurrence} holds for $j_0=0$ and $j_0=1$ and
	$(\xi_1,\xi_2)\neq(0,0)$, then $a$ is a root of the polynomial $\a_0+\a_1 x+\ldots +\a_l x^l$.
	\item If $a$ is a double root of the polynomial $\a_0+\a_1 x+\ldots +\a_l x^l$, then
	equation \eqref{eq:recurrence} holds for all $j_0\in\Z_{\ge 0}$ and all $\xi_1,\xi_2\in \F_q$.
\end{enumerate}
\end{lemma}

\begin{proof}
We begin with the first claim.
If $\xi_2\neq 0$, we subtract $a$ times equation \eqref{eq:recurrence} for
$j_0=0$ from the same equation for $j_0=1$.
We get
\[
\sum_{j=0}^l\a_j(\xi_1(a^{j+1}-a^{j+1})+\xi_2((j+1)a^j-ja^j))=0,
\]
which reduces to
\[
\sum_{j=0}^l\a_j\xi_2a^j=0,
\]
and proves the claim upon dividing the equation by $\xi_2$.

If $\xi_2=0$ and then necessarily $\xi_1\neq 0$, we get the claim if we divide
\eqref{eq:recurrence} for $j_0=0$ by $\xi_1$.
This proves the first claim. 

Next, we turn to the second claim.
Using that $a$ is a root of the polynomial $x^{j_0}(\a_0+\ldots+\a_lx^l)$ and of its derivative,
we get the equations
\begin{align*}
\sum_{j=0}^l\a_j a^{j+j_0}&=0,\\
\sum_{j=0}^l(j+j_0)\a_j a^{j+j_0-1}&=0.
\end{align*}
Taking a linear combination of these equations with coefficients
$\xi_1, \xi_2$, we get \eqref{eq:recurrence}.
This proves the second claim.
\end{proof}

Let $X=(x_0,\ldots,x_N)$ be a sequence of integers.
We write $\Lambda(X)$ for the set of polynomials $P(x)=\a_0+\ldots+\a_d x^d$
of degree $d$ for some $d\le N$ such that
\[
\sum_{j=0}^d \a_j x_{j+j_0}=0
\]
for all $j_0=0,\ldots, N-d$.

The next lemma, which we quote from \cite{Kon}*{Lemma 5}, will be used to show that
the sequence $b_j$ satisfies the linear recurrence relation corresponding
to the greatest common divisor of the two polynomials that we will construct.

\begin{lemma}\label{lm:Konyagin}
Let $X=(x_0,\ldots,x_N)$ be an integer sequence.
Suppose $P_1,P_2\in\Lambda(X)$ and $\deg(P_1)+\deg(P_2)\le N$
then we have $\gcd(P_1,P_2)\in\Lambda(X)$.
\end{lemma}

One of the polynomials that we will construct will be a polynomial in $x^r$
for a suitable number $r$.
The next lemma gives an estimate for the Mahler measures of low degree
divisors of such a polynomial.
This result is standard, but we include the proof for the sake of completeness.

\begin{lemma}\label{lm:Mahler}
Let $P_1,P_2\in\Z[x]$ be non-zero polynomials and let $r>\deg(P_1)^2$ be a prime
and suppose $P_1|P_2$, where $P_2(x)=Q(x^r)$ for some $Q\in\Z[x]$.
Then $M(P_1)^r\le M(P_2)$.
\end{lemma}

\begin{proof}
Let $\zeta$ be a primitive $r$'th root of unity.
For each root $\a$ of $P_1$, $\a\zeta^j$ is a root of $P_2$ for all
$j=0,\ldots,r-1$ with the same multiplicity as $\a$ is a root of $P_1$.
Moreover, these are distinct numbers because $\deg(\a/\a')< \deg(P_1)^2$
for any two roots $\a,\a'$ of $P_1$ and $\deg(\zeta^j)\ge r-1$ for
$j=1,\ldots,r-1$.

Therefore, denoting the places of $\Q$ by $M(\Q)$, we have
\begin{align*}
M(P_2)&=\prod_{v\in M(\Q)}\prod_{\a\in\overline{\Q_v}: P_2(\a)=0} \max(1,|\a|_v)\\
&\ge \prod_{v\in M(\Q)}\prod_{\a\in\overline{\Q_v}:P_1(\a)=0}
\prod_{j=0}^{r-1} \max(1,|\a\zeta^j|_v)\\
&=M(P_1)^r.
\end{align*}
\end{proof}

\begin{proof}[Proof of Lemma \ref{lm:large-element}]
If $a=0$, the conclusion holds, so we assume that $a\neq 0$.
We observe that
\begin{align*}
\xi_1 a^{jk+j_0}+\xi_2(jk+j_0)a^{jk+j_0-1}
&= \xi_1a^{j_0}(a^k)^j+\xi_2(jk+j_0)a^{j_0-1}(a^k)^{j}\\
&=(\xi_1 a^{j_0}+\xi_2 j_0 a^{j_0-1})(a^k)^j
+\xi_2k a^{k+j_0-1}j(a^k)^{j-1}.
\end{align*}
If we replace $a$ by $a^k$, $\xi_1$ by $\xi_1 a^{j_0}+\xi_2 j_0 a^{j_0-1}$
and $\xi_2$ by $\xi_2k a^{k+j_0-1}$,
this reduces the lemma to the case $j_0=0$ and $k=1$, so we will only consider that case.
It is easy to check that the new values of $\xi_1$ and $\xi_2$ are
not both $0$ if the original values were not.

Let $b_j$ be the smallest integer in absolute value  in the residue
class of $\xi_1a^j+\xi_2ja^{j-1}$ for $j\in\Z_{\ge 0}$.
By assumption, we have
\[|b_j|<\frac{q^{1-2/(l+1)}}{l+1}\]
for all $j=0,\ldots, l(r+1)$.

Let $\a_0,\ldots,\a_l, \b_0,\ldots,\b_l\in\Z$
with $|\a_j|,|\b_j|<q^{2/(l+1)}$ for all $j$ be
such that
$a$ is a double root of both polynomials
$P_1(x)=\a_0+\ldots+\a_l x^l$ and $P_2(x)=\b_0+\b_1 x^r+\ldots+\b_l x^{rl}$,
and both $P_1$ and $P_2$ are non-zero.
To find the coefficients of $P_1$, we consider the family of polynomials
\[
\{\wt\a_0+\ldots+\wt \a_l x^l:\wt \a_j=0,\ldots,\lfloor q^{2/(l+1)}\rfloor\}.
\]
By the pigeonhole principle, there will be two among them, say $R_1$ and $R_2$,
such that $R_1(a)=R_2(a)$ and $R_1'(a)=R_2'(a)$, and we take $P_1=R_1-R_2$.
We may find $P_2$ analogously.

We assume, as we may, that $\a_l$ and $\b_l$ are both non-zero, for otherwise we
could multiply $P_1$ or $P_2$ by suitable powers of $x$.

Now by Lemma \ref{lm:recurrence}, it follows that
\begin{equation}\label{eq:cong1}
\sum_{j=0}^{l}\a_j b_{j+j_0}\equiv 0\mod q
\end{equation}
for all $j_0=0,\ldots, lr$
and
\begin{equation}\label{eq:cong2}
\sum_{j=0}^{l}\b_j b_{jr+j_0}\equiv 0\mod q
\end{equation}
for all $j_0=0,\ldots,l$.
For the second claim, we use the lemma for a coefficient sequence of length
$lr+1$ with $r-1$ zeroes inserted between consecutive $\beta_j$.

By the triangle inequality and the upper bounds on the $\a_j$, $\b_j$ and $b_j$, we
have that the left hand sides of \eqref{eq:cong1} and \eqref{eq:cong2}
are less than
\[
(l+1)\frac{q^{1-2/(l+1)}}{l+1}q^{2/(l+1)}=q
\]
in absolute value, hence they are $0$.
Therefore, $P_1,P_2\in \Lambda(b_0,\ldots,b_{l(r+1)})$.

By Lemma~\ref{lm:Konyagin}, we have $\gcd(P_1,P_2)\in\Lambda(b_0,\ldots, b_{l(r+1)})$.
By Lemma~\ref{lm:recurrence}, $a$ is then a root of $\gcd(P_1,P_2)$.
We clearly have $\deg(\gcd(P_1,P_2))\le\deg(P_1)\le l$ and
\[
M(\gcd(P_1,P_2))\le M(P_2)^{1/r}\le ((l+1)q^{2/(l+1)})^{1/r}\le q^{1/(l+1)^2}.
\]
The first inequality follows from Lemma \ref{lm:Mahler}.
The second inequality follows from the standard bound for Mahler measure
in terms of the $l^1$ norm of the coefficients.
For the last inequality, we use that $l\ge 3$ and $q$ is sufficiently large in terms of $l$.
Therefore, $P: = \gcd(P_1,P_2)$ is an $(l,q)$-exceptional polynomial
and $a$ is an $l$-exceptional residue.
\end{proof}

\begin{proof}[Proof of Proposition \ref{pr:unexceptional}]
We consider the Fourier transform of the distribution of $Y$, which we
compute as
\begin{align*}
\wh Y(\xi_1,\xi_2)&=\E[\exp(2\pi i(\xi_1,\xi_2)\cdot Y/q)]\\
&=\prod_{j\in I}\E[\exp(2\pi iX_j(\xi_1 a^j+\xi_2 j a^{j-1})/q)]\\
&=\prod_{j\in I}\cos(2\pi(\xi_1 a^j+\xi_2 j a^{j-1})/q).
\end{align*}

We note the elementary inequality
\[
\cos(2\pi b/q)\le \exp(-(\pi^2/2)(|2b|/q)^2).
\]
Suppose $(\xi_1,\xi_2)\neq(0,0)$.
We apply Lemma \ref{lm:large-element} with a prime $r\in[l^2, 2l^2]$ and conclude that
every arithmetic progression in $\Z_{\ge 0}$
of length
\[
l(r+1)+1\le l(2l^2+1)\le 3l^3
\]
with common difference at most $K$
contains an element $j$ such that
\[
|2\xi_1 a^j+2\xi_2 j a^{j-1}|\ge\frac{q^{1-2/(l+1)}}{l+1},
\]
and hence
\[
\cos(2\pi(\xi_1 a^j+\xi_2 j a^{j-1})/q)\le \exp(-(\pi^2/2)q^{-4/(l+1)}/(l+1)^2).
\]

Since there are more than
\[
q^{5/(l+1)}\ge 10\log q\cdot (2/\pi^2) q^{4/(l+1)}(l+1)^2
\] 
disjoint arithmetic progressions in $I$, we have
\[
|\wh Y(\xi_1,\xi_2)|\le q^{-10}
\]
provided $(\xi_1,\xi_2)\neq(0,0)$.

Using the Fourier inversion formula
\[
\P[Y=x]=\frac{1}{q^2}\sum_{\xi_1,\xi_2\in\F_q}
\exp(-2\pi i(\xi_1,\xi_2)\cdot x)\wh Y(\xi_1,\xi_2)
\]
and $\wh Y(0,0)=1$, we conclude
\[
\Big|\P[Y=x]-\frac{1}{q^2}\Big|\le q^{-10}
\]
as required.
\end{proof}

We finish this section by recording a classical Littlewood-Offord type
bound that shows that the random walk spreads out substantially even when
$a$ is an exceptional residue.
In the case $a=1$, the result is essentially sharp.

\begin{proposition}\label{pr:exceptional}
Let $q$ be an odd prime, and let $a\in\F_q^\times$.
Let $I$ be a set of positive integers with $|I|\le q$.
Let $X_i$ be independent uniform $\pm1$ valued random variables for $i\in I$, and let
\[
Y=\sum_{i\in I} X_i a^i
\]
be a random element of $\F_q$.
Then
\[
\P[Y=x]\le 129 |I|^{-1/2}.
\]
for all $x\in \F_q$.
\end{proposition}

\begin{proof}
	Let $\wt X_i$ for $i\in I$ be a sequence of independent random variables
	taking the value $0$ with probability $3/4$ and each of $\pm 1$ with
	probability $1/8$.
	
	By \cite{CJMS}*{Lemma 12} applied with $\mu=1/4$, we have
	\[
	\P\Big[\sum_{i\in I} \wt X_i a^i =x\Big]\le 64(|I|/4)^{-1/2}+q^{-1}
	\le 129 |I|^{-1/2}
	\]
	for all $x\in \F_q$.
	
	By \cite{TV-book}*{Corollary 7.12} applied with $\mu'=1$ and $\mu=1/4$, we have
	\[
	\P\Big[\sum_{i\in I} X_i a^i =x\Big]\le \P\Big[\sum_{i\in I} \wt X_i a^i =0\Big]
	\le 129 |I|^{-1/2}
	\]
	for all $x\in \F_q$.
\end{proof}

\section{Expected number of roots of a random polynomial in a finite field}
\label{sc:number-of-roots}

Recall that $X_j$ for $j\in\Z_{\ge 0}$ is a random multiplicative sequence.
To simplify notation, we introduce the random polynomial
\[R(x): = 
P_{d+1}(x)/x=X_1+X_2x+\ldots+X_{d+1} x^d.
\]
Recall also that $B_Q(q)$ is the number of distinct roots of a polynomial $Q$
in the finite field $\F_q$.

The purpose of this section is to deduce the following estimates for the expected number of roots and double roots of $R$.

\begin{proposition}\label{pr:number-roots-random-poly}
Fix $l\in\Z_{>0}$ and $\e\in(0,1/10)$.
Let $q$ be an odd prime and $d\in\Z_{>0}$ that are both sufficiently
large in terms of $l$ and such that $l\ge \e^{-1}\log q/\log d$ and $q>d$.
Then
\begin{align*}
|\E[B_R(q)]-1|&<d^{-1/2+\e},\\
\E[|\{a\in\F_q: \text{$a$ is a double root of $R$}\}|]&<d^{-1/2+\e}.
\end{align*}
\end{proposition}

We comment on the role of the parameter $l$ in the above statement and in the proof.
We will need to apply this proposition such that $q$ is a sufficiently large
power of $d$ so that the error term in the prime ideal theorem
(or in Proposition \ref{pr:number-distinct-factors} to be precise)
is smaller than our claimed estimate in the main theorem.
In fact, $q> d^5$ will be sufficient for this purpose.
We can then set $l$ depending on $\e$ and $\log q/\log d$.
For the proposition to hold, we need that $d$ and $q$ are large enough depending on $l$.
This is because $l$ ultimately controls the length of the arithmetic progressions
we need in the set $I$, so we need $d$ to be large enough to satisfy
the condition of the Green Tao theorem.

The proof follows easily from Propositions \ref{pr:unexceptional}
and \ref{pr:exceptional} and the
Theorem of Green and Tao \cite{GT} that the primes contain arbitrarily long arithmetic
progressions.
We will use the following version that provides control for the step size
of the progressions, which we recall from \cite{TZ}*{Theorem 5} and
\cite{Sha}*{Theorem 1.3}.

\begin{theorem}\label{th:Green-Tao}
For every $L\in\Z_{>0}$, there is a constant $C=C(L)$ such that the following
holds for all $d\in\Z_{>0}$ that is sufficiently large in terms of $L$.
Let $A$ be a subset of the primes in $[1,d]$ with $|A|>d/10\log d$.
Then $A$ contains an arithmetic progression of length $L$ with common difference
less than $C(\log d)^C$.
\end{theorem}

We also need the following simple lemma to count the number of exceptional residues
for which Proposition \ref{pr:exceptional} will be applied.

\begin{lemma}\label{lm:number-exceptional}
For every $l$, there is a constant $C=C(l)$ such that the number of
$(l,q)$-exceptional polynomials is at most
$Cq^{1/(l+1)}$.
\end{lemma}

\begin{proof}
The coefficients of a polynomial $Q\in\Z[x]$
of degree at most $l$
are bounded by $C M(Q)$ for some constant $C$ depending only on $l$.
Therefore, an $(l,q)$-exceptional polynomial has coefficients bounded by
$Cq^{1/(l+1)^2}$ due to the definition of being $(l, q)$-exceptional. The claim follows by raising this bound to the power $l+1$.
\end{proof}

\begin{proof}[Proof of Proposition \ref{pr:number-roots-random-poly}]
Let $K=C_1(\log d)^{C_1}$,
where $C_1$ is the constant $C$
in Theorem \ref{th:Green-Tao}
applied with $L=3l^3$.
We first consider the probability that some $a\in\F_q$
is a root or a double root of $R$
under the condition that $a^k$ is not $l$-exceptional for any $1\le k\le K$.
To this end, we will apply Proposition \ref{pr:unexceptional} with the set of
primes in $[d/2,d]$ in the role of $I$.

We first show that $I$ contains the required number of arithmetic progressions.
We apply Theorem \ref{th:Green-Tao} repeatedly to find arithmetic progressions
of length $3l^3$ with common difference at most $K$.
We start this with all primes in $[d/2,d]$ in the role of $A$ at first, then
we remove from $A$ all arithmetic progressions that we find to apply
Theorem \ref{th:Green-Tao} again for this reduced set.

This process can be run more than $d/(10\log(d) l^3)$ times before the number
of elements in $A$ falls below the threshold in the theorem, so
we find at least this many arithmetic progressions.
Since $l\ge 5\log q/\log d$,
\[
q^{5/(l+1)}\le \frac{d}{10\log(d) l^3}, 
\]
and we have enough arithmetic progressions to apply Proposition~\ref{pr:unexceptional}.

Next we write
\[
(R(a),R'(a))= Y+Z,
\]
where
\[
Y=\sum_{i\in I}X_{i}(a^{i-1},(i-1)a^{i-2})
\]
and
\[
Z=\sum_{i\in [1,d+1]\backslash I}X_{i}(a^{i-1},(i-1)a^{i-2}).
\]
We observe that $Y$ and $Z$ are independent. So, conditioning on the value of $Z$, we can write
\[
\Big|\P[R(a)=x_1,R'(a)=x_2]-\frac{1}{q^2}\Big|
\le \max_{z\in\F_q^2}\Big|\P[Y=(x_1,x_2)-z]-\frac{1}{q^2}\Big|
< q^{-10}
\]
for all $x_1,x_2\in\F_q$, where we applied Proposition~\ref{pr:unexceptional}.
We conclude
\begin{align}
\Big|\P[R(a)=0]-\frac{1}{q}\Big|&\le q^{-9}\label{eq:unexceptional1}\\
\P[\text{$a$ is a double root of $R$}]\le \frac{2}{q^2}.\label{eq:unexceptional2}
\end{align}

Now we prove a bound that is valid for all $a\in\F_q^{\times}$.
We write $R(a)=Y_1+Z_1$, where $Y_1$ and $Z_1$ are the first components
of the vectors $Y$ and $Z$ defined above.
Conditioning on the value of $Z_1$ and then using
Proposition \ref{pr:exceptional} we can write
\[
\P[R(a)=0]\le \max_{z\in\F_q}\P[Y_1=-z]\le 129 |I|^{-1/2}\le 1000 (d/\log d) ^{-1/2}.
\]
Note that $R(0)=1$ almost surely, so the probability that $0$ is a root is $0$, so
the above bound is valid even for $a=0$.

The number of elements $a\in\F_q$ for which the bounds \eqref{eq:unexceptional1}
and \eqref{eq:unexceptional2} do not apply is at most
$K^2lC_2q^{1/(l+1)}$, where $C_2$ is the constant
$C$ in Lemma \ref{lm:number-exceptional}.
We see that the expected number of exceptional roots of $R$ is less than
\[
1000C_1^2C_2l(\log d)^{2C_1+1/2}q^{1/(l+1)} d^{-1/2}<d^{-1/2+\e}/2
\]
because $l\ge \e^{-1}\log q/\log d$.

Summing up \eqref{eq:unexceptional2} for $a\in\F_q$ that are not exceptional
and combining with the above bound, we get
\[
\E[|\{a\in\F_q: \text{$a$ is a double root of $R$}\}|]\le \frac{2}{q}+d^{-1/2+\e}/2
< d^{-1/2+\e}.
\]

Summing up \eqref{eq:unexceptional1} and combining with the bound for exceptional
roots, we get
\[
\Big|\E[|\{a\in\F_q: R(a)=0\}|]-\frac{q}{q}\Big|
\le \frac{K^2lC_2q^{1/(l+1)}}{q}+q^{-8}+ d^{-1/2+\e}/2<d^{-1/2+\e}.
\]
\end{proof}

\section{Proof of Theorem \ref{thm: random}}
\label{sc:proof-thm}

Fix some $\e>0$, and an integer $l>5\e^{-1}$.
Let $d$ be sufficiently large in terms of $l$, and let
$X=5\log d$. Recall $R(x)=P_{d+1}(x)/x$. 
Proposition~\ref{pr:number-roots-random-poly} gives
\begin{equation}\label{eq:BPq}
|\E[B_R(q)]-1|< d^{-1/2+\e}
\end{equation}
for all primes $q$ with $\log q\in[X-\log 2,X]$.

Applying Proposition \ref{pr:prime-number-theorem} for $K=\Q$, we get
\begin{equation}\label{eq:weights}
\sum_{\text{$q$ prime}} \log (q) h_X(\log q)= 1+ O(X^2\exp(-X/2))
\end{equation}
Summing up \eqref{eq:BPq}, we get
\begin{align*}
\sum_{\text{$q$ prime}} B_R(q)\log(q) h_X(\log q)
&=1+O(X^2\exp(-X/2))+O(d^{-1/2+\e})\\
&=1+O(d^{-1/2+\e}).
\end{align*}

Now we apply Proposition \ref{pr:number-distinct-factors} with $M=1$
and get
\begin{align*}
\E[|\{\text{distinct} &\text{ irreducible factors of $R$}\}|]\\
&=1+O(d^{-1/2+\e})+ O(d^2 X^2\log(d)\exp(-X/2))\\
&=1+O(d^{-1/2+\e}).
\end{align*}
Since $R$ has always at least $1$ irreducible factor, Markov's inequality
implies that the probability that it has more than $1$ is less than
$C d^{-1/2+\e}$.

It remains to prove that $R$ is not a proper power of a single irreducible polynomial
with high probability.
Write $\wt B_R(q)$ for the number of elements of $\F_q$ that are roots of
$R$ with multiplicity at least $2$.
When $R$ is a proper power, $B_R(q)=\wt B_R(q)$.

Using that $R$ always has at least one irreducible factor, it follows
from Proposition \ref{pr:number-distinct-factors} that
\[
\sum_{\text{$q$ prime}} B_R(q)\log(q)h_X(\log q) \ge 1/2
\]
for any fixed realization of $R$.
Therefore, we have 
\[
\sum_{\text{$q$ prime}} \E[\wt B_R(q)]\log(q)h_X(\log q)
\ge\P[\text{$R$ is a proper power}]/2,
\]
because the left hand side is the expectation of a random variable that is at least $1/2$ when $R$
is a proper power.
Using Proposition \ref{pr:number-roots-random-poly} and \eqref{eq:weights},
we get
\[
2 d^{-1/2+\e}>\P[\text{$R$ is a proper power}]/2,
\]
as required.

\section{Proof of Corollary~\ref{cr:pd}}
\label{sc:proof-cor}

We deduce the corollary from the main theorem and the following result
about the distribution of Legendre symbols due to Granville and Soundararajan.

\begin{theorem}[\cite{GS}*{Proposition 9.1}]\label{th:GS}
Suppose the Riemann hypothesis holds for all Dirichlet $L$-functions.
Let $x,d\in\Z_{>0}$.
Let $\omega_r=\pm1$ for each prime $r\le d$, and let
$\cP(x,\{\omega_r\})$ denote the set of primes $p\le x$ such that
$\big(\frac{r}{p}\big)=\omega_r$ for each $r\le d$.
Then for $d\le x^{1/2}$, we have
\[
\sum_{p\in\cP(x,\{\omega_r\})}\log p = \frac{x}{2^{\pi(d)}}+O(x^{1/2}(d+\log x)^2).
\]
\end{theorem}

\begin{proof}[Proof of Corollary \ref{cr:pd}]
In the statement of the corollary, the prime $p$ is selected from $[d+1,f(d)]$
uniformly.
We modify the distribution of $p$ in such a way that $p$ is selected from
$[1,f(d)]$ with probability
\begin{equation}\label{eq:modified-distribution}
\frac{\log p}{\sum_{p\in[1,f(d)]}\log p}\le C \frac{\log p}{f(d)}.
\end{equation}
This does not affect the the validity of the conclusion for the following reason.
The probability that $p<f(d)^{1/10}$ is smaller than $C d^{-9/10}$ in both models,
which is much smaller than the claimed bound for the probability that $F_{d,f}$ is reducible,
so we may ignore this event.
Conditioning on $p> f(d)^{1/10}$ the ratio between the probabilities that $p$ takes a given value
in the two models is bounded between two absolute constants, so the probability of
the event of reducibility will also only change by a constant factor.
For the rest of the proof, we consider $p$ to be random with distribution \eqref{eq:modified-distribution}.

To prove the corollary, it is enough to show that for any fixed
$\{\omega_r\}\in\{-1,1\}^{\text{$r\le d$ prime}}$ the probability that
$\big(\frac{r}{p}\big)=\omega_r$ for all $r$ is bounded by an absolute
constant times
\[
\P[\text{$X_r=\omega_r$ for all $r$}]=2^{-\pi(d)}.
\]
Therefore, it is enough to show that
\[
\sum_{p\in\cP(f(d),\{\omega_r\})}\frac{\log p}{f(d)} \le \frac{C}{2^{\pi(d)}}
\]
for some constant $C>0$.
This clearly follows from Theorem \ref{th:GS}.
\end{proof}

 \bibliographystyle{abbrv}
	\bibliography{Fekete}{}
\end{document}